\documentclass[11pt]{article}

\usepackage{latexsym}
\usepackage{amssymb}
\usepackage{amsthm}
\usepackage{amscd}
\usepackage{amsmath}
\usepackage[all]{xy}
\usepackage{graphicx}

\newtheorem{theorem}{Theorem}[section]

\newtheorem{lemma}{Lemma}[section]

\newtheorem{corollary}{Corollary}[section]

\xyoption{dvips}

\numberwithin{equation}{section}

\newcommand{\FF}{\mathbb{F}}

\newcommand{\CC}{\mathbb{C}}

\def\dim{\mathrm{dim}} 
\def\Hom{\mathrm{Hom}}

\def\Ind{\mathrm{Ind}} 
\def\GL{\mathrm{GL}}

\def\SL{\mathrm{SL}}

\def\veps{\varepsilon}

\def\Hom{\mathrm{Hom}}

\def\diag{\mathrm{diag}}

\newcommand{\Sp}{\mathrm{Sp}}

\makeatletter
\renewcommand{\@makefnmark}{\mbox{\textsuperscript{}}}
\makeatother

\def\adots{\mathinner{\mkern2mu\raise0pt\hbox{.}  
\mkern2mu\raise4pt\hbox{.}\mkern1mu
\raise7pt\vbox{\kern7pt\hbox{.}}\mkern1mu}}

\allowdisplaybreaks[1]

\begin{document}

\bibliographystyle{amsplain}

\title{Klyachko models of $p$-adic special linear groups}
\author {Joshua M. Lansky and C. Ryan Vinroot}
\date{}

\maketitle



\begin{abstract}
We study Klyachko models of $\SL(n, F)$, where $F$ is a nonarchimedean local field.  In particular, using results of Klyachko models for $\GL(n, F)$ due to Heumos, Rallis, Offen and Sayag, we give statements of existence, uniqueness, and disjointness of Klyachko models for admissible representations of $\SL(n, F)$, where the uniqueness and disjointness are up to specified conjugacy of the inducing character, and the existence is for unitarizable representations in the case $F$ has characteristic $0$.  We apply these results to relate the size of an $L$-packet containing a given representation of $\SL(n, F)$ to the type of its Klyachko model, and we describe when a self-dual unitarizable representation of $\SL(n, F)$ is orthogonal and when it is symplectic.
\end{abstract}

\section{Introduction\protect\footnote{2000 {\em Mathematics Subject Classification. } Primary 22E50}\protect\footnote{{\em Key words and phrases. } Klyachko model, $p$-adic special linear group, multiplicity one, $L$-packets, self-dual representations }}

Let $F$ be a field, let $U_m(F)$ denote the group of $m$-by-$m$ unipotent upper triangular matrices over $F$, and let $M_{m,l}(F)$ be the set of $m$-by-$l$ matrices over $F$ (not necessarily invertible).  For each integer $k$ satisfying $0 \leq 2k \leq n$, define the subgroup $G_k$ of $\GL(n, F)$ by:
\begin{equation} \label{GkDefn}
G_k = \left\{ \left( \begin{array}{cc} N & X \\   & S \end{array} \right) \Big| N \in U_{n-2k}, S \in \Sp(2k, F), X \in M_{n-2k, 2k}(F) \right \}.
\end{equation}
Fix a nontrivial additive character $\theta: F^{+} \rightarrow \CC$, and for each $k$, define a character $\psi_k$ on $G_k$ as follows:
\begin{equation}  \label{psidefn}
\text{If } g \in G_k, g = \left( \begin{array}{cc} N & X \\   & S \end{array} \right), \text{ and } N = (a_{ij}), \text{ then define } \psi_k(g) = \theta\left( \sum_{i=1}^{n-2k-1} a_{i, i+1} \right).
\end{equation}
In other words, $\psi_k$ is only non-trivial on the unipotent factor of $G_k$.  When $n = 2m$, then $\psi_m$ is just the trivial character on the subgroup $G_m = \Sp(2m, F)$, and when $k=0$, $\psi_k$ is a nondegenerate character of the unipotent subgroup $U_n(F)$ of $\GL(n, F)$.

Suppose that $F = \FF_q$ is a finite field, let $G = \GL(n, \FF_q)$, and for each $k$, $0 \leq 2k \leq n$, define the induced representation $T_k = \Ind_{G_k}^G(\psi_k)$.  Klyachko \cite{Kla83} proved that for any complex irreducible representation $(\pi, V)$ of $G$, $\dim_{\CC} \, \Hom_G(\pi, T_k) \leq 1$ for every $k$, and there exists a unique $k$, $0 \leq 2k \leq n$, such that $\dim_{\CC} \, \Hom_G(\pi, T_k) = 1$.

We call an embedding of the representation $(\pi, V)$ in the induced representation $T_k$ a {\em Klyachko model} of the representation $\pi$.  Klyachko's original result states that every irreducible representation of $\GL(n, \FF_q)$ has a unique Klyachko model, and in particular, all of the induced representations $T_k$ are multiplicity-free, and $T_k$ and $T_l$ have no isomorphic sub-representations when $k \neq l$.

Now consider the case that $F$ is a nonarchimedean local field, with $G = \GL(n, F)$.  For each $k$, $0 \leq 2k \leq n$, define the representation $T_k$ by
$$ T_k = \Ind_{G_k}^{G} (\psi_k),$$
where $\Ind$ denotes the ordinary (non-normalized) induced representation for a locally compact totally disconnected group.  In this case, there is the following result on Klyachko models of representations of $\GL(n, F)$.

\begin{theorem}[Heumos and Rallis, Offen and Sayag] \label{LocalCase}

Let $G = \GL(n, F)$, where $F$ is a nonarchimedean local field.  Let $(\pi, V)$ be any irreducible admissible representation of $G$.  We have the following:
\begin{enumerate}
\item[(1)]   $\displaystyle{\sum_{k=0}^{\lfloor n/2\rfloor}\dim_{\CC} \, \Hom_H(\pi, T_k) \leq 1.}$
\item[(2)] If $F$ has characteristic $0$ and $(\pi,V)$ is unitarizable, then there exists a unique $k$ such that $\dim_{\CC} \, \Hom_G(\pi, T_k) = 1$.
\end{enumerate}
\end{theorem}

Heumos and Rallis \cite{HeuRal90} proved that, if $n = 2m$, then for any $\pi$, $\dim_{\CC} \, \Hom_H(\pi, T_{m}) \leq 1$; that is, any irreducible admissible representation has a unique {\em symplectic model} if one exists.  They also proved that in this case, the set of admissible representations of $\GL(n, F)$ which have symplectic models is disjoint with the set of representations which have Whittaker models.  Finally, Heumos and Rallis proved statements (1) and (2) of Theorem \ref{LocalCase} for $n \leq 4$ and conjectured that these statements hold for all $n$.  Theorem \ref{LocalCase} was proved completely by Offen and Sayag in a series of papers \cite{OffSay07,OffSay081,OffSay082}.

Now notice that the groups $G_k$ are also subgroups of the special linear group $\SL(n, F)$.  In this paper, we study Klyachko models of the group $\SL(n, F)$ when $F$ is a nonarchimedean local field.  Since there is more than one orbit of nondegenerate characters of the unipotent subgroup of $\SL(n, F)$, we must consider conjugates of the characters $\psi_k$ in (\ref{psidefn}) in these models.  Our main result, Theorem \ref{main}, is the analogue of Theorem \ref{LocalCase} for the special linear group.  The main difference in the result is in the statement of uniqueness and disjointness of Theorem \ref{main}, where we can only obtain uniqueness and disjointness of Klyachko models up to conjugation of the character $\psi_k$ by a certain group.

We give two applications of Theorem \ref{main}.  In the first, Corollary \ref{Lpacket}, we relate the type of the Klyachko model of a representation of $\SL(n, F)$ to the size of the $L$-packet containing that representation.  In the second, Corollary \ref{SLnselfdual}, we describe when a self-dual unitarizable representation of $\SL(n, F)$ is orthogonal and when it is symplectic.
\\
\\
\noindent {\bf Acknowledgments. } The first-named author was supported by NSF grant DMS-0854844, and the second-named author was supported by NSF grant DMS-0854849.

\section{Klyachko models of special linear groups}

From now on, we let $F$ be a nonarchimedean local field, let $G = \GL(n, F)$, let $H = \SL(n, F)$, and let $G_k$ be as in (\ref{GkDefn}) for each $k$ such that $0 \leq 2k \leq n$.  Note that $G\cong H\ltimes D$, where $D\cong F^\times$ is the group of matrices of the form $\mathrm{diag} (x,1,\ldots 1)$ for $x\in F^\times$.  We will often identify $G/H$ with $D$ and hence with $F^\times$.  Note that $D$ normalizes each $G_k$, and $H$ contains each $G_k$.  The open subgroup $HZ$ of $G$ is normal, and $G/HZ\cong D/D^n\cong F^\times/(F^\times)^n$.  In particular, $HZ$ has finite index in $G$.  

Let $(\rho ,W)$ be a representation of $H$.  By~\cite[Prop.~2.2]{Tad92}, there is an $H$-embedding of $(\rho ,W)$ as a direct summand of some irreducible admissible representation $(\pi,V)$ of $G$.  Then $W$ is stable under the action of $Z$; hence we can view $\rho$ as a representation of $HZ$.  From results in \cite{BusKut93,Tad92}, we know that if $\rho$ is a unitarizable then $(\pi,V)$ can also be taken to be unitarizable.

Let $(\rho,W)$ be a representation of $H$.  Given any $g\in G$, define $^g\rho$ to be the representation of $H$ on $W$ given by $^g \rho (h) = \rho (g^{-1}hg)$.  Denote by $G(\rho)$ the subgroup $\{ g\in G\vert {}^g \rho\cong\rho\}$ of $G$.  We note that $G(\rho)$ contains $HZ$, hence is of finite index in $G$.  We let $D(\rho) = D\cap G(\rho)$.

If $g\in G$ normalizes $G_k$, and $\psi$ is a character of $G_k$, denote by $^g \psi$ the character $\alpha\mapsto \psi (g^{-1}\alpha g)$.  For every $x\in F^{\times}$, and $k$ such that $0 \leq 2k \leq n$, define a character ${^x \psi_k}$ on $G_k$, with the same notation as in (\ref{psidefn}), by
$$ {^x \psi_k}(g) = \theta \left(xa_{1,2} + \sum_{i=2}^{n-2k-1} a_{i, i+1} \right).$$
If $g\in G$ normalizes $G_k$ and has image $x\in F^\times$ under the map $G\mapsto G/H\cong D \cong F^\times$, we have that $^g \psi = {}^x\psi$.  We first prove a lemma relating models for representations of $G$ with those of $H$.

\begin{lemma}
\label{embed}
Let $(\rho,W)$ be an irreducible admissible representation of $H$ and let $(\pi,V)$ be an irreducible admissible representation of $G$ that contains $(\rho,W)$ upon restriction.  Suppose that for some $\gamma\in D$ and some $k$ with $0\leq 2k\leq n$, $(\rho,W)$ embeds in $\Ind_{G_k}^H({}^\gamma\psi_k)$.  Then $(\pi,V)$ embeds in $T_k$.
\end{lemma}
\begin{proof}
Since $\rho$ embeds in $\Ind_{G_k}^H({}^\gamma\psi_k)$, we have by Frobenius reciprocity,
$$\Hom_{HZ}(\rho,\Ind_{G_k}^{HZ}({}^\gamma\psi_k))\neq (0),$$
where we view $\rho$ as a representation of $HZ$ as discussed above.
Inducing to $G$, it follows that
$$\Hom_{G}(\Ind_{HZ}^G(\rho),\Ind_{G_k}^{G}({}^\gamma\psi_k))\neq (0).$$
This implies that one of the constituents $\pi_0$ of $\Ind_{HZ}^G(\rho)$ embeds in $T_k$.  
By~\cite[Cor.~2.5]{Tad92}, $\pi_0$ is isomorphic to a twist of $\pi$ by a one-dimensional character of $G$.  Since characters are trivial on $G_k$, it follows that $\pi$ also embeds in $T_k$.
\end{proof}

We now prove our main result.
\begin{theorem}
\label{main}
Let $H = \SL(n, F)$, where $F$ is a nonarchimedean local field.  Let $(\rho, W)$ be an irreducible admissible representation of $H$.  We have the following:
\begin{enumerate}
\item[(1)] For any $\gamma\in D$,
\begin{equation}
\label{unique}
\sum_{k=0}^{\lfloor n/2\rfloor}\dim_{\CC} \, \Hom_H(\rho, \Ind_{G_k}^H({}^\gamma\psi_k) ) \leq 1.
\end{equation}
Moreover, if this sum is nonzero for some $\gamma\in D$, then such a $\gamma$ is unique modulo $D(\rho)$.
\item[(2)] If $F$ has characteristic $0$ and $\rho$ is unitarizable, then there exists a unique integer $k$ and $\gamma\in D$ (unique modulo $D(\rho)$) such that $\dim_{\CC} \, \Hom_H(\rho, \Ind_{G_k}^H({^\gamma \psi_k})) = 1$.
\end{enumerate}
\end{theorem}

\begin{proof}
Let $(\rho ,W)$ be an irreducible admissible representation of $H$ and let $\gamma\in D$.  Let $(\pi,V)$ be an irreducible admissible representation of $G$ in which $(\rho,W)$ embeds as a direct summand.  If $\pi$ has no Klyachko model, then $\Hom_H(\rho, \Ind_{G_k}^H({}^\gamma\psi_k))$ must be trivial for all integers $k$ by Lemma~\ref{embed}, so (\ref{unique}) holds.  Hence suppose from now on that $\pi$ embeds in $T_k$ for some integer $k$ with $0\leq 2k\leq n$.

Viewing $\rho$ as a representation of $HZ$ as above, Mackey's theorem~\cite[Exer. 4.5.5]{Bump97} implies that we have an isomorphism
\begin{equation}
\label{mackey}
\Hom_G \left(\Ind_{HZ}^G (\rho),T_k\right) \cong \bigoplus_{\delta\in G/HZ}\Hom_{G_k} ({}^\delta\rho,\psi_k).
\end{equation}
A straightforward argument using Mackey's theorem shows that $\Ind_{HZ}^G (\rho)$ is the direct sum of $(G(\rho):HZ)$ irreducible admissible representations of $G$, each of which is obtained from $\pi$ via twisting by an appropriate one-dimensional character of $G$~\cite[Cor.~2.5, Prop.~2.7]{Tad92}.  Since characters of $G$ are trivial on $G_k$, each of these representations occurs with multiplicity one in $T_k$ since $\pi$ does.  Thus the dimension of the space on the left-hand side of (\ref{mackey}) is $(G(\rho):HZ)$.

Now consider the right-hand side of (\ref{mackey}).  As $\delta$ ranges over $G/HZ$, $^\delta\rho$ ranges over $(G:G(\rho))$ distinct representations of $G$, each one occurring $(G(\rho):HZ)$ times.  Hence the right-hand side of (\ref{mackey}) is a direct sum of $(G(\rho):HZ)$ copies of 
$$\bigoplus_{\delta\in G/G(\rho)}\Hom_{G_k} ({}^\delta\rho,\psi_k).$$
Together with the preceding paragraph, this implies that
\begin{equation}
\label{mult_one}
\sum_{\delta\in G/G(\rho)}\dim_\CC\left(\Hom_{G_k} ({}^\delta\rho,\psi_k)\right) = 1.
\end{equation}
Note that
$$\Hom_{G_k} ({}^\delta\rho,\psi_k) = \Hom_{G_k} (\rho,{}^{\delta^{-1}}\psi_k) =  \Hom_{H} (\rho,\Ind_{G_k}^H({}^{\delta^{-1}}\psi_k)).$$
Also note that we may assume that our representatives for the cosets in $G/G(\rho)$ lie in $D$.  Thus we can rewrite (\ref{mult_one}) to obtain
\begin{equation*}
\sum_{\delta\in D/D(\rho)}\dim_\CC\left(\Hom_{H} (\rho,\Ind_{G_k}^H({}^\delta\psi_k))\right) = 1.
\end{equation*}
This implies that there is a $\gamma\in D$, unique modulo $D(\rho)$, such that $\Hom_{H} (\rho,\Ind_{G_k}^H({}^\gamma\psi_k))$ is nontrivial.  This forces $\Hom_{H} (\rho,\Ind_{G_k}^H({}^\gamma\psi_k))$ to be one-dimensional.

Now suppose $\rho$ also embeds in $\Ind_{G_k}^H({}^\delta\psi_l)$ for some integer $l$ and $\delta\in D$.  Then Lemma~\ref{embed} implies that $\pi$ also embeds in $T_l$, which forces $l=k$ by the uniqueness of the Klyachko model of $\pi$.
This concludes the proof of (1) in the case that $\pi$ has a Klyachko model, and shows that in this case (\ref{unique}) is an equality.

Now suppose that $\rho$ is unitarizable.  Note that statement (2) now follows from (1) as soon as it is shown that the representation $\pi$ has a Klyachko model.  But by~\cite[Prop.~2.2, 2.7]{Tad92}, we may assume that $\pi$ is itself unitarizable.  Hence by Theorem~\ref{LocalCase}, $\pi$ has a Klyachko model.
\end{proof}

We will say that the representation $\rho$ of $H$ possesses a Klyachko model if $\Hom_H(\rho, \Ind_{G_k}^H({}^\gamma\psi_k))$ is nontrivial for some integer $k$ and some $\gamma\in D$.  Note that Theorem \ref{main} can be adjusted to be a statement for Klyachko models for the finite group $\SL(n, \FF_q)$, which sharpens the results in \cite[Prop. 1]{Vin06}.  

We will need the following for an application of Theorem \ref{main}.

\begin{lemma}
\label{stabilizer}
Let $k$ be an integer, $0\leq 2k\leq n$ and let $d = (2k,n)$.  Suppose $\psi$ is a character of $G_k$ that is trivial on $\left\{ \left( \begin{array}{cc} 1_{n-2k} & X \\   & 1_{2k} \end{array} \right) \Big| X \in M_{n-2k, 2k}(F) \right\}$.  Then the equivalence class of $\Ind_{G_k}^H(\psi)$ is stable under conjugation by $D^d$.
\end{lemma}
\begin{proof}
Suppose $\delta\in D^n$.  Since $\det\delta$ is an $n$th power, $\delta\in HZ$.  Thus
$${}^\delta(\Ind_{G_k}^H(\psi)) \cong \Ind_{G_k}^H(\psi).$$
Thus $D^n$ stabilizes the equivalence class of $\Ind_{G_k}^H(\psi_k)$.

Now suppose $\delta\in D^{n-2k}$ so that $\delta = {\rm diag}(a^{n-2k},1,\ldots ,1)$ for some $a\in F^\times$.  
Let $\alpha = {\rm diag}(a, \ldots ,a,1,\ldots ,1)\in G$, where  the blocks of $a$'s and $1$'s have respective lengths $n-2k$ and $2k$.  Note that $\delta\in\alpha H$ and that conjugation by $\alpha$ fixes $\psi$.  Thus
$${}^\delta(\Ind_{G_k}^H(\psi)) = {}^\alpha(\Ind_{G_k}^H(\psi)) = \Ind_{G_k}^H({}^\alpha\psi) = \Ind_{G_k}^H(\psi).$$
Therefore, $D^{n-2k}$ stabilizes the equivalence class of $\Ind_{G_k}^H(\psi)$.

It follows from the preceding paragraphs that the group generated by $D^n$ and $D^{n-2k}$ stabilizes the equivalence class of $\Ind_{G_k}^H(\psi)$.  To complete the proof, note that this group is precisely $D^d$.
\end{proof}

The Local Langlands Correspondence for $\GL(n)$~\cite{HarTay01,Hen00} gives a bijection from the set of equivalence classes of irreducible representations of $G$ to a set consisting of certain $n$-dimensional complex representations of the Weil-Deligne group $W_F'$ of $F$.  The existence of the Langlands Correspondence for $\SL(n)$ follows from this by the work of Gelbart and Knapp~\cite{GelKna82}.  Here the equivalence classes of irreducible representations of $H$ are parameterized by certain homomorphisms from $W_F'$ to ${\rm PGL}(n, \CC)$.  Moreover, in the case of $\SL(n)$, the correspondence is now many-to-one; the fibers of the parameterization are the \textit{$L$-packets} of $H$.  In~\cite[Thm.~4.1]{GelKna82}, it is shown that the $L$-packets of $H$ coincide with the orbits of $G$ on equivalence classes of irreducible representations of $H$.  Thus if $\rho$ is an irreducible admissible representation of $H$, the size of the $L$-packet containing $\rho$ is precisely $(G:G(\rho)) = (D:D(\rho))$.  The following result gives a relationship between the Klyachko model of a representation $\rho$ of $H$ and the size of the $L$-packet containing $\rho$.

\begin{corollary} \label{Lpacket}
Let $k$ be an integer, $0\leq 2k\leq n$.  Let $d = (2k,n)$.  If the irreducible admissible representation $\rho$ of $H$ occurs in $\Ind_{G_k}^H({}^\gamma \psi_k)$ for some $\gamma\in D$, then $D^d\subset D(\rho)$.  In particular, the size of the $L$-packet of $\rho$ is at most the index of $(F^\times)^d$ in $F^\times$.  Thus if $d=1$, then $\rho$ must be stable, that is, the $L$-packet containing $\rho$ is a singleton.
\end{corollary}

\begin{proof}
Recalling that $D\cong F^\times$, the second and third statements follow immediately from the first, which we now verify.  It follows from Lemma~\ref{stabilizer} that $D^d$ stabilizes $\Ind_{G_k}^H({}^\gamma \psi_k)$.  Let $\delta\in D^d$.  Then $\rho$ occurs in $\Ind_{G_k}^H({}^{\gamma\delta} \psi_k)$.  By the uniqueness statement in Theorem~\ref{main}, we must then have that $\gamma\delta\in\gamma D(\rho)$ so $\delta\in D(\rho)$.
\end{proof}
 
\section{Self-dual representations}

Let $G$ be a totally disconnected locally compact group with $(\pi, V)$ an irreducible admissible representation of $G$, and $\iota$ a continuous automorphism of $G$ such that $\iota^2$ is the identity.  Let $(\hat{\pi}, \hat{V})$ denote the smooth contragredient of $(\pi, V)$, where $\hat{V}$ is the smooth dual of $V$, and define the representation $({^\iota \pi}, V)$ by ${^\iota \pi} = \pi \circ \iota$.  From Schur's Lemma, the representation $\pi$ satisfies ${^\iota \pi} \cong \hat{\pi}$ if and only if there exists a nondegenerate bilinear form, unique up to scalar multiple, say $B: V \times V \rightarrow \CC$, such that
\begin{equation} \label{Bform}
B(\pi(g)v, {^\iota \pi}(g)w) = B(v, w) \, \, \text{for all } v, w \in V, g \in G.
\end{equation}
It follows that $B$ must be either symmetric, in which case we write $\veps_{\iota}(\pi) = 1$, or skew-symmetric, in which case we write $\veps_{\iota}(\pi) = -1$.  If ${^\iota \pi} \not\cong \hat{\pi}$, then we let $\veps_{\iota}(\pi) = 0$.  When $\iota$ is the trivial automorphism, then ${^\iota \pi} = \pi \cong \hat{\pi}$ just means that $\pi$ is self-dual.  In this case, we simply write $\veps(\pi)$ for $\veps_{\iota}(\pi)$.  If $\pi$ is self-dual and $\veps(\pi) = 1$, we say $\pi$ is {\em orthogonal}, and if $\veps(\pi) = -1$, we say $\pi$ is {\em symplectic}.

We begin with the following, which is a slight generalization of \cite[Lemma 2.1]{PrRa08}.  Since the proof is virtually identical to the proof in \cite{PrRa08}, we just give an outline.

\begin{lemma} \label{selfduallemma1}
Let $(\pi, V)$ be an irreducible, admissible, and unitarizable representation of the totally disconnected locally compact group $G$, and let $\iota$ be a continuous automorphism of $G$ such that $\iota^2$ is the identity.  Then $\veps_{\iota}(\pi) = 1$ if and only if there exists a conjugate linear automorphism $\varphi: V \rightarrow V$ such that $\varphi^2 = 1$, and $\varphi({^\iota \pi}(g)v) = \pi(g) \varphi(v)$ for all $v \in V$ and all $g \in G$.
\end{lemma}
\begin{proof}  Since $(\pi, V)$ is unitarizable, there is a positive definite Hermitian form $\langle \cdot, \cdot \rangle$ on $V$ which is $G$-invariant.  First assume there exists a conjugate linear automorphism $\varphi$ on $V$ with the above properties.  If we define a bilinear form $B$ by $B(v, w) = \langle v, \varphi(w) \rangle$, then it follows that $B$ is nondegenerate and satisfies (\ref{Bform}).  To prove that $B$ is symmetric, it is enough to show that $\langle v, w \rangle = \overline{\langle \varphi(v), \varphi(w) \rangle}$, which follows from the uniqueness of $\langle \cdot, \cdot \rangle$ up to positive scalar multiple.

Conversely, suppose that $B$ is a nondegenerate symmetric form on $V$ which satisfies (\ref{Bform}).  Any element of the smooth dual $\hat{V}$ of $V$ is of the form $\langle \cdot, w \rangle$, for a unique $w \in V$.  For any $w \in V$, the map $u \mapsto B(u, w)$ is a smooth linear functional of $V$, and so there is a unique $w'$ such that $B(u,w) = \langle u, w' \rangle$.  This defines a conjugate linear map $w \mapsto w'$ on $V$.  Now, we must have $\langle v, w \rangle = \lambda \overline{\langle v', w' \rangle}$, for all $v, w \in V$ and for some positive real number $\lambda$, by uniqueness of the Hermitian form $\langle \cdot, \cdot \rangle$.  If we define $\varphi(v) = \sqrt{\lambda} v'$, then $\varphi:V \rightarrow V$ has the desired properties. 
\end{proof}

The next result is a generalization of \cite[Cor. 2.2]{PrRa08}, and we again use an argument very similar to the one appearing there.

\begin{lemma} \label{selfduallemma2}
Let $(\pi, V)$ be an irreducible, admissible, and unitarizable representation of $G$, let $\iota$ be a continuous automorphism of $G$ such that $\iota^2$ is the identity, and let $H$ be a closed subgroup of $G$ which is stable under $\iota$.  Let $\psi$ be a one-dimensional representation of $H$ such that ${^\iota \psi} = \bar{\psi}$, and such that $\dim_{\CC} \, \Hom_H (\pi, \psi) = 1$.  If ${^\iota \pi} \cong \hat{\pi}$, then $\veps_{\iota}(\pi) = 1$.
\end{lemma}
\begin{proof} Let $\langle \cdot, \cdot \rangle$ denote the $G$-invariant Hermitian form on $V$.  We know that ${^\iota \pi} \cong \hat{\pi}$, and say $T: V \rightarrow \hat{V}$ is the corresponding intertwining operator.  There is also a conjugate linear isomorphism $L: V \rightarrow \hat{V}$ given by $L(w) = \langle \cdot, w \rangle$, and note that $L$ satisfies $L(\pi(g)v) = \hat{\pi}(g) L(v)$ for all $g \in G$, $v \in V$.  Then $\eta = L^{-1} \circ T$ is a conjugate linear automorphism of $V$ satisfying $\eta({^\iota \pi}(g) v) = \pi(g) \eta(v)$ for all $g \in G$, $v \in V$.  By Schur's lemma, we must have $\eta^2 = \alpha$, where $\alpha$ is some nonzero complex scalar.

Now, let $\ell \in \Hom_H(\pi, \psi)$, and define $\tilde{\ell}: V \rightarrow \CC$ by $\tilde{\ell} (v) = \overline{\ell(\eta(v))}$.  Then, for any $h \in H$, $v \in V$, we have
$$ \tilde{\ell}(\pi(h) v) = \overline{\ell({^\iota \pi}(h) \eta(v))} = \overline{ {^\iota \psi}(h) \ell(\eta(v))} = \psi(h) \tilde{\ell}(v),$$
since $\eta(\pi(h) v) = {^\iota \pi}(h) \eta(v)$ and ${^\iota \psi} = \bar{\psi}$.  So, $\tilde{\ell} \in \Hom_H(\pi, \psi)$, and we must have $\tilde{\ell} = \lambda \ell$ for some nonzero complex scalar $\lambda$.  Since we then have $\ell(\eta(v)) = \overline{\lambda \ell(v)}$ for all $v$, then by substituting $\eta(v)$ for $v$, and from the fact $\eta^2(v) = \alpha v$, we obtain $\alpha \ell(v) = \bar{\lambda} \lambda \ell(v)$.  We now have $\alpha = \bar{\lambda} \lambda$, and we define $\varphi = \lambda^{-1} \eta$.  Now, $\varphi: V \rightarrow V$ is a conjugate linear automorphism such that $\varphi^2 = 1$ and $\varphi({^\iota \pi}(g) v) = \pi(g) \varphi(v)$ for all $g \in G$, $v \in V$.  By Lemma \ref{selfduallemma1}, we have $\veps_{\iota}(\pi) = 1$.
\end{proof}

If $(\pi, V)$ is an irreducible admissible representation of $G$, and $z$ is an element of the center of $G$, then it follows from Schur's lemma that $\pi(z)$ acts as a scalar on $V$, which we denote by $\omega_{\pi}(z)$.  The next result follows directly from \cite[Prop. 2]{Vin06}.

\begin{lemma} \label{selfduallemma3}
Let $s \in G$ such that $s^2 = z$ is in the center of $G$.  Define the automorphism $\iota$ on $G$ by $\iota(g) = s^{-1} g s$, so $\iota^2$ is the identity.  Then for any irreducible admissible representation $(\pi, V)$ of $G$, we have $\veps(\pi) = \omega_{\pi}(z) \veps_{\iota}(\pi)$.
\end{lemma}

In \cite[Sec. 3, Ex. (2)]{Pr99}, Prasad describes when a generic self-dual representation of $\SL(n, F)$ is orthogonal and when it is symplectic (excluding the case that $n$ is $2$ mod $4$ and $F$ does not contain a square root of $-1$).  Here, we extend these results to include any self-dual irreducible admissible representation which is unitarizable.

\begin{corollary} \label{SLnselfdual}
Let $F$ be a nonarchimedean local field of characteristic $0$, and let $(\pi, V)$ be a self-dual, irreducible, admissible, and unitarizable representation of $H = \SL(n, F)$.  Then
\begin{enumerate}
\item[(1)]  If $n$ is odd or $n \equiv 0($mod $4)$, then $\varepsilon(\pi) = 1$.
\item[(2)]  If $n \equiv 2($mod $4)$ and $F$ contains a square root of $-1$, then $\varepsilon(\pi) = 1$ if and only if the central element $-I$ of $\SL(n, F)$ acts trivially on $V$, that is, $\veps(\pi) = \omega_{\pi}(-I)$.
\end{enumerate}
\end{corollary}
\begin{proof} By Theorem \ref{main}(2), there exists a $k$, $0 \leq 2k \leq n$, and a $\gamma \in D$, such that
$$ \dim_{\CC} \, \Hom_G(\pi, \Ind_{G_k}^H({^\gamma \psi_k})) = \dim_{\CC} \, \Hom_{G_k} (\pi, {^\gamma \psi_k}) = 1.$$
If $n \equiv 0({\rm mod}\; 4)$, then define $s = \diag(-1, 1, \ldots, -1, 1)$; if $n \equiv 3({\rm mod} \;4)$ then define $s = \diag(-1, 1,\ldots, 1, -1)$; and if $n \equiv 1({\rm mod} \; 4)$, then define $s = \diag(1, -1,\ldots, -1, 1)$.  Then $s \in H$, $s^2 = I$, and if we define $\iota$ on $H$ by $\iota(g) = s^{-1} g s$, then $G_k$ is stable under $\iota$.  We have ${^\gamma \psi_k}(s^{-1} h s) = \overline{{^\gamma \psi_k}(h)}$ for every $h \in G_k$, and ${^\iota \pi} \cong \pi \cong \hat{\pi}$, since $\pi$ is self-dual.  By Lemma \ref{selfduallemma2}, we have $\veps_{\iota}(\pi) = 1$, and by Lemma \ref{selfduallemma3} we have $\veps(\pi) = 1$, as desired.

Now suppose that $n \equiv 2({\rm mod} \; 4)$, and that $F$ contains a square root of $-1$, and say $\beta \in F$ such that $\beta^2 = -1$.  Define $s = \diag(\beta, -\beta,\ldots, \beta, -\beta)$, and define $\iota$ on $H$ by $\iota(g) = s^{-1} g s$.  Then $s^2 = -I$, and $G_k$ is stable under $\iota$.  Like before, we have ${^\gamma \psi_k}(s^{-1} h s) = \overline{{^\gamma \psi_k}(h)}$ for every $h \in G_k$, and also ${^\iota \pi} \cong \hat{\pi}$.  By Lemmas \ref{selfduallemma2} and \ref{selfduallemma3}, we conclude that $\veps(\pi) = \omega_{\pi}(-I)$.
\end{proof}

\noindent
{\bf Remarks. } In \cite[Sec. 6]{Vin06}, the second-named author studies the values of $\veps_{\iota}(\pi)$, where $\pi$ is an irreducible admissible representation of $\GL(n, F)$, and $\iota$ is the transpose-inverse automorphism composed with conjugation by the longest Weyl element.  The statement in \cite[Thm. 8]{Vin06} that $\veps_{\iota}(\pi) = 1$ for all such $\pi$ does not have a complete proof there.  What is actually proved is that if $\pi$ is an irreducible admissible representation of $\GL(n, F)$, and there exists a character $\psi$ of the maximal unipotent subgroup such that ${^\iota \psi} = \bar{\psi}$ and $\pi$ has a unique $\psi$-degenerate Whittaker model, then $\veps_{\iota}(\pi) = 1$.  Also, the conclusion cannot be made in \cite[Sec. 3]{Vin06} using similar methods that $\veps_{\iota}(\pi) = 1$ for every irreducible representation $\pi$ of the finite group $\GL(n, \FF_q)$.  However, this statement is already known to be true for the finite group $\GL(n, \FF_q)$, while this is still an open question for the $p$-adic group $\GL(n, F)$.

For the statement in \cite[Thm. 8]{Vin06} that $\veps(\pi) = 1$ for every self-dual, irreducible, admissible representation $\pi$ of $\GL(n, F)$, the proof is complete.  It is possible that similar methods could be used to extend Corollary \ref{SLnselfdual} to all self-dual irreducible admissible representations of $\SL(n, F)$.

\bigskip

\noindent
\begin{tabular}{ll}
\textsc{Department of Mathematics and Statistics}\\
\textsc{American University}\\
\textsc{4400 Massachusetts Avenue, NW}\\
\textsc{Washington, DC  20016}\\
{\em email}:  {\tt lansky@american.edu}\\
\end{tabular}

\bigskip

\noindent
\begin{tabular}{ll}
\textsc{Department of Mathematics}\\ 
\textsc{College of William and Mary}\\
\textsc{P. O. Box 8795}\\
\textsc{Williamsburg, VA  23187}\\
{\em email}:  {\tt vinroot@math.wm.edu}\\
\end{tabular}


\begin{thebibliography}{10}

\bibitem{Bump97}
D.~Bump, Automorphic forms and representations, Cambridge Studies in Advanced Mathematics, 55, Cambridge University Press, Cambridge, 1997.

\bibitem{BusKut93}
C.J.~Bushnell and P.C.~Kutzko, The admissible dual of ${\rm SL}(N)$. I, \emph{Ann. Sci. \'{E}cole Norm. Sup.} (4) \textbf{26} (1993), 261--279.

\bibitem{GelKna82}
S.S.~Gelbart and A.W.~Knapp, $L$-indistinguishability and $R$ groups for the special linear group, \emph{Adv. Math.} \textbf{43} (1982), no. 2, 101--121.

\bibitem{HarTay01}
M.~Harris and R.~Taylor, The geometry and cohomology of some simple Shimura varieties, with an appendix by Vladimir G. Berkovich, Ann.~of~Math.~Studies, 151, Princeton University Press, Princeton, New Jersey, 2001.

\bibitem{Hen00}
G.~Henniart, Une preuve simple des conjectures de Langlands pour $\GL(n)$ sur un corps $p$-adique, \emph{Invent.~Math.} \textbf{139} (2000), no. 2, 439--455.

\bibitem{HeuRal90}
M.J.~Heumos and S.~Rallis, Symplectic-Whittaker models for $Gl_n$, \emph{Pacific J. Math.} \textbf{146} (1990), no. 2, 247--297.

\bibitem{Kla83}
A.A.~Klyachko, Models for complex representations of the groups $\GL(n,q)$, \emph{Mat. Sb. (N.S.)}, \textbf{120(162)} (1983), no. 3, 371--386.

\bibitem{OffSay07}
O.~Offen and E.~Sayag, On unitary representations of ${\rm GL}(2n)$ distinguished by the symplectic group, \emph{J. Number Theory} \textbf{125} (2007), no. 2, 344--355.

\bibitem{OffSay081}
O.~Offen and E.~Sayag, Global mixed periods and local Klyachko models for the general linear group, \emph{Int. Math. Res. Not. IMRN} (2008), Art. ID rnm 126, 25 pages. 

\bibitem{OffSay082}
O.~Offen and E.~Sayag, Uniqueness and disjointness of Klyachko models, \emph{J. Funct. Anal.} \textbf{254} (2008), no. 11, 2846--2865.

\bibitem{Pr99}
D.~Prasad, On the self-dual representations of a $p$-adic group, \emph{Int. Math. Res. Not. IMRN} (1999), no. 8, 443--452.

\bibitem{PrRa08}
D.~Prasad and D.~Ramakrishnan, Self-dual representations of division algebras and Weil groups:  a contrast, preprint, 2008, available at {\tt www.math.tifr.res.in/$\sim$dprasad} and {\tt www.math.caltech.edu/people/index.html}.

\bibitem{Tad92}
M.~Tadi\`{c}, Notes on representations of non-Archimedean ${\rm SL}(n)$, \emph{Pacific J. Math.} \textbf{152} (1992), no. 2, 375--396.

\bibitem{Vin06}
C.R.~Vinroot, Involutions acting on representations, \emph{J. Algebra} \textbf{297} (2006), no. 1, 50--61.

\end{thebibliography}
\end{document}